\documentclass{ejpamPRE}

\begin{document}

\title{Application of the Inclusion-Exclusion Principle to Prime Number Subsequences}

\author[M.P. May] {Michael P. May}

\address{Department of Mechanical and Aerospace Engineering, Missouri University of Science and Technology, Rolla, Missouri, USA\\}

\emails{{\tt mike.may.bbi@gmail.com}}

\begin{abstract}
We apply the Inclusion-Exclusion Principle to a unique pair of prime number subsequences to determine whether these subsequences form a small set or a large set and thus whether the infinite sum of the inverse of their terms converges or diverges. In this paper, we analyze the complementary prime number subsequences $\mathbb{P}'$ and $\mathbb{P}''$ as well as revisit the twin prime subsequence $\mathbb{P}_2$.
\end{abstract}


\maketitle

\section{The Prime Subsequences ${\mathbb{P{'}}}$ and ${\mathbb{P{''}}}$}

\noindent The prime number subsequence \cite{OEIS_A333242}

\begin{equation*}
\mathbb{P^{'}} = \left\lbrace {p{'}} \right\rbrace = \left\lbrace 2,5,7,13,19,23,29,31,37,43,47,53,59,61,71,... \right\rbrace
\end{equation*}

\noindent can be generated via an alternating sum of the prime number subsequences of increasing order \cite{May2020}, i.e., 

\begin{equation}
{\mathbb{P^{'}}={\left\lbrace{(-1)^{n-1}}\left\lbrace{p^{(n)}}\right\rbrace\right\rbrace}_{n=1}^\infty} \label{eq:1}
\end{equation}

\noindent where the right-hand side of Eq. \ref{eq:1} is an expression of the alternating sum 

\begin{equation}
\left\lbrace{p^{(1)}}\right\rbrace - \left\lbrace{p^{(2)}}\right\rbrace + \left\lbrace{p^{(3)}}\right\rbrace - \left\lbrace{p^{(4)}}\right\rbrace + \left\lbrace{p^{(5)}}\right\rbrace -... \;. \label{eq:2}
\end{equation} \par

\noindent The prime number subsequences of increasing order \cite{BroughanBarnett2009} in Expression \ref{eq:2} are defined as

\newpage

\begin{equation*}
\left\lbrace{p^{(1)}}\right\rbrace = {\left\lbrace{p_n}\right\rbrace}_{n=1}^\infty = \left\lbrace2,3,5,7,11,13,17,19,23,29,31,37,41,43,...\right\rbrace 
\end{equation*}

\begin{equation*}
\left\lbrace{p^{(2)}}\right\rbrace = {\left\lbrace{p_{p_n}}\right\rbrace}_{n=1}^\infty = \left\lbrace3,5,11,17,31,41,59,67,83,109,127,...\right\rbrace 
\end{equation*}

\begin{equation*}
\left\lbrace{p^{(3)}}\right\rbrace = {\left\lbrace{p_{p_{p_n}}}\right\rbrace}_{n=1}^\infty = \left\lbrace5,11,31,59,127,179,277,331,...\right\rbrace 
\end{equation*}

\begin{equation*}
\left\lbrace{p^{(4)}}\right\rbrace = {\left\lbrace{p_{p_{p_{p_n}}}}\right\rbrace}_{n=1}^\infty = \left\lbrace11,31,127,277,709,...\right\rbrace 
\end{equation*}

\begin{equation*}
\left\lbrace{p^{(5)}}\right\rbrace = {\left\lbrace{p_{p_{p_{p_{p_n}}}}}\right\rbrace}_{n=1}^\infty = \left\lbrace31,127,709,...\right\rbrace 
\end{equation*}

\noindent and so on and so forth. Thus, the operation performed on the right-hand side of Eq. \ref{eq:1} denotes an infinite alternating sum of the sets of prime number subsequences of increasing order.

\

\noindent The prime number subsequence $\mathbb{P^{'}}$ can also be generated by performing a structured alternating summation of the individual elements across the sets denoted on the right-hand side of Eq. \ref{eq:1}. To illustrate this, we arrange the subsequences of increasing order \cite{BroughanBarnett2009} in Expression \ref{eq:2} side-by-side and sum elements laterally across the rows to create the $\mathbb{P^{'}}$ subsequence term-by-term as follows:

\begin{table}[ht]
\centering
\caption{Alternating Sum of $p^{(n)}$}
\label{table:sum}
\begin{tabular}{cccccccccc}
(row) & $+p^{(1)}$ & $-p^{(2)}$ & $+p^{(3)}$ & $-p^{(4)}$ & $+p^{(5)}$ & $-p^{(6)}$ & $...$ & $p{'}$ \\ 
&  &  &  &  &  &  &  &  &  \\ 
(1) & 2 & $\longrightarrow$ & $\longrightarrow$ & $\longrightarrow$ & $\longrightarrow$ & $\longrightarrow$ & $\longrightarrow$ &  2 \\ 
(2) & 3 & 3 & $\longrightarrow$ & $\longrightarrow$ & $\longrightarrow$ & $\longrightarrow$ & $\longrightarrow$ &  0 \\ 
(3) & 5 & 5 & 5 & $\longrightarrow$ & $\longrightarrow$ & $\longrightarrow$ & $\longrightarrow$ &  5 \\ 
(4) & 7 & $\longrightarrow$ & $\longrightarrow$ & $\longrightarrow$ & $\longrightarrow$ & $\longrightarrow$ & $\longrightarrow$ &  7 \\ 
(5) & 11 & 11 & 11 & 11 & $\longrightarrow$ & $\longrightarrow$ & $\longrightarrow$ &  0 \\ 
(6) & 13 & $\longrightarrow$ & $\longrightarrow$ & $\longrightarrow$ & $\longrightarrow$ & $\longrightarrow$ & $\longrightarrow$ &  13 \\ 
(7) & 17 & 17 & $\longrightarrow$ & $\longrightarrow$ & $\longrightarrow$ & $\longrightarrow$ & $\longrightarrow$ &  0 \\ 
(8) & 19 & $\longrightarrow$ & $\longrightarrow$ & $\longrightarrow$ & $\longrightarrow$ & $\longrightarrow$ & $\longrightarrow$ &  19 \\ 
(9) & 23 & $\longrightarrow$ & $\longrightarrow$ & $\longrightarrow$ & $\longrightarrow$ & $\longrightarrow$ & $\longrightarrow$ &  23 \\ 
(10) & 29 & $\longrightarrow$ & $\longrightarrow$ & $\longrightarrow$ & $\longrightarrow$ & $\longrightarrow$ & $\longrightarrow$ &  29 \\ 
(11) & 31 & 31 & 31 & 31 & 31 & $\longrightarrow$ & $\longrightarrow$ &  31 \\ 
$\vdots$ & $\vdots$ & $\vdots$ & $\vdots$ & $\vdots$ & $\vdots$ & $\vdots$ & $\vdots$ & $\vdots$  
\end{tabular}
\end{table}

\noindent Thus, the infinite prime number subsequence $\mathbb{P^{'}}$ of higher order \cite{OEIS_A333242} that emerges in the rightmost column of Table \ref{table:sum} is

\begin{equation*}
{\mathbb{P^{'}} = \left\lbrace {p{'}} \right\rbrace} = \left\lbrace 2,5,7,13,19,23,29,31,37,43,47,53,59,61,71,... \right\rbrace.
\end{equation*}

\noindent The prime number subsequence $\mathbb{P^{'}}$ can also be generated by performing a sieving operation on the natural numbers $\mathbb{N}$ \cite{May2020}. Starting with $n=1$, choose the prime number with subscript 1 (i.e., $p_1=2$) as the first term of the subsequence and eliminate that prime number from the natural number line.  Next, proceed forward on $\mathbb{N}$ from $1$ to the next available natural number.  Since $2$ was eliminated from the natural number line in the previous step, one moves forward to the next available natural number that has not been eliminated, which is $3$.  The prime number $3$ then becomes the subscript for the next $\mathbb{P^{'}}$ term which is $p_3=5$, and $5$ is then eliminated from the natural number line, and so on and so forth.  Such a sieving operation has been carried out in Table \ref{table:sieve} for the natural numbers $1$ to $100$: \\

\begin{table}[ht]
\centering
\caption{Sieving $\mathbb{N}$ to generate $\mathbb{P}'$}
\label{table:sieve}
\begin{tabular}{ccccccccccc}
&  &  &  &  &  &  &  &  &  & \\ 
1 & {\textbf{2}} & 3 & 4 & {\textbf{5}} & 6 & {\textbf{7}} & 8 & 9 & 10 \\ 
11 & 12 & {\textbf{13}} & 14 & 15 & 16 & 17 & 18 & {\textbf{19}} & 20 \\ 
21 & 22 & {\textbf{23}} & 24 & 25 & 26 & 27 & 28 & {\textbf{29}} & 30 \\ 
{\textbf{31}} & 32 & 33 & 34 & 35 & 36 & {\textbf{37}} & 38 & 39 & 40 \\ 
41 & 42 & {\textbf{43}} & 44 & 45 & 46 & {\textbf{47}} & 48 & 49 & 50 \\ 
51 & 52 & {\textbf{53}} & 54 & 55 & 56 & 57 & 58 & {\textbf{59}} & 60 \\ 
{\textbf{61}} & 62 & 63 & 64 & 65 & 66 & 67 & 68 & 69 & 70 \\
{\textbf{71}} & 72 & {\textbf{73}} & 74 & 75 & 76 & 77 & 78 & {\textbf{79}} & 80 \\ 
81 & 82 & 83 & 84 & 85 & 86 & 87 & 88 & {\textbf{89}} & 90 \\ 
91 & 92 & 93 & 94 & 95 & 96 & {\textbf{97}} & 98 & 99 & 100 \\ 
\end{tabular}
\end{table}

\noindent Thus, we may also designate $\mathbb{P{'}}$, which has been alternately created via the sieving operation in Table \ref{table:sieve}, by the following notation \cite{May2020} to indicate that the natural numbers $\mathbb{N}$ have been sieved to produce this prime number subsequence:

\begin{equation*}
{\left\lfloor{\raisebox{-.3pt}{\!{{\begin{math}\, \underline{\mathbb{N}}\end{math}}}}}\right\rfloor} = {\mathbb{P{'}}} = \left\lbrace{2,5,7,13,19,23,29,31,37,43,47,53,59,61,71,...}\right\rbrace.
\end{equation*}

\noindent Regardless of which one of these three methods is used to generate $\mathbb{P{'}}$, when the prime numbers in this unique subsequence are applied as indexes to the set of all prime numbers $\mathbb{P^{}}$, one obtains the next higher-order prime number subsequence $\mathbb{P{''}}$ \cite{OEIS_A262275}:

\begin{equation*}
{\mathbb{P^{''}} = \left\lbrace {p{''}} \right\rbrace} = \left\lbrace3,11,17,41,67,83,109,127,157,191,211,241,...\right\rbrace. 
\end{equation*}
 
\noindent By definition, the sequence $\mathbb{P{''}}$ can also be generated via the expression \cite{May2020}

\begin{equation}
{\mathbb{P^{''}}={\left\lbrace{(-1)^{n}}\left\lbrace{p^{(n)}}\right\rbrace\right\rbrace}_{n=2}^\infty} \label{eq:3}
\end{equation}

\noindent where an expansion of the right-hand side of Eq. \ref{eq:3} reveals the alternating sum 

\begin{equation*}
\left\lbrace{p^{(2)}}\right\rbrace - \left\lbrace{p^{(3)}}\right\rbrace + \left\lbrace{p^{(4)}}\right\rbrace - \left\lbrace{p^{(5)}}\right\rbrace + \left\lbrace{p^{(6)}}\right\rbrace -... \;. \label{eq:4}
\end{equation*}

\noindent The prime number subsequence of higher order $\mathbb{P^{''}}$ can also be generated by performing the aforementioned sieving operation on the set of all prime numbers $\mathbb{P}$, similar to how the primes $\mathbb{P^{'}}$ were sifted from the set of all natural numbers $\mathbb{N}$. Furthermore, it has been shown \cite{May2020} that the subsequences $\mathbb{P^{'}}$ and $\mathbb{P^{''}}$, when added together, form the entire set of prime numbers $\mathbb{P^{}}$:

\begin{equation}
{\mathbb{P{}}}={\mathbb{P{'}}+\mathbb{P{''}}}. \label{eq:5}
\end{equation}

\noindent We sketch a proof of Eq. \ref{eq:5} here:

\

\begin{proof}

\noindent It has been shown \cite{May2020} that 

\begin{equation*}
{\mathbb{P{'}}} = {\left\lbrace{(-1)^{n-1}}\left\lbrace{p^{(n)}}\right\rbrace\right\rbrace}_{n=1}^\infty = \left\lbrace{p^{(1)}}\right\rbrace - \left\lbrace{p^{(2)}}\right\rbrace + \left\lbrace{p^{(3)}}\right\rbrace - ... 
\end{equation*}

\noindent and

\begin{equation*}
{\mathbb{P{''}}} = {\left\lbrace{(-1)^{n}}\left\lbrace{p^{(n)}}\right\rbrace\right\rbrace}_{n=2}^\infty = \left\lbrace{p^{(2)}}\right\rbrace - \left\lbrace{p^{(3)}}\right\rbrace + \left\lbrace{p^{(4)}}\right\rbrace - ... \; .
\end{equation*}

\noindent Therefore,

\begin{align*}
{\mathbb{P{'}}+\mathbb{P{''}}} = & \left\lbrace{p^{(1)}}\right\rbrace - \left\lbrace{p^{(2)}}\right\rbrace + \left\lbrace{p^{(3)}}\right\rbrace - ... \\
& \;\;\;\;\;\;\;\;\;\;\;\;\;\;\;\;\;\;\;\; + \\
& \left\lbrace{p^{(2)}}\right\rbrace - \left\lbrace{p^{(3)}}\right\rbrace + \left\lbrace{p^{(4)}} \right\rbrace - ... = \left\lbrace{p^{(1)}}\right\rbrace = {\mathbb{P{}}}.
\end{align*}

\end{proof}

\noindent An interesting property was observed in the relationship between the set of all prime numbers $\mathbb{P{}}$ and the complementary prime number sets $\mathbb{P{'}}$ and $\mathbb{P{''}}$. Since $\mathbb{P{''}}=\mathbb{P}_{\mathbb{P{'}}}$, Eq. \ref{eq:5} can be rewritten as

\begin{align*}
{\mathbb{P{''}}} & {=\mathbb{P{}} - \left\lbrace2,5,7,13,19,23,29,...\right\rbrace} \\
& {=\left\lbrace {p_{{}_{{}_{{{2}}}}},p_{{}_{{}_{{{5}}}}},p_{{}_{{}_{{{7}}}}},p_{{}_{{}_{{{13}}}}},p_{{}_{{}_{{{19}}}}},p_{{}_{{}_{{{23}}}}}, p_{{}_{{}_{{{29}}}}},...}\right\rbrace = \mathbb{P}_{\mathbb{P{'}}}}
\end{align*}

\noindent where the prime numbers of the subsequence $\mathbb{P{'}}$ form the indexes for the complement set of primes $\mathbb{P{''}}$ such that 

\begin{equation*}
{{\mathbb{P{''}}} = {\mathbb{P}_{{\mathbb{P{'}}}} = \lbrace p_k \mid k \in \mathbb{P{'}} \rbrace}}.
\end{equation*}

\section{Asymptotic Densities of ${\mathbb{P{'}}}$ and ${\mathbb{P{''}}}$}

\noindent We will now derive the asymptotic densities for the prime number subsequences $\mathbb{P{'}}$ and $\mathbb{P{''}}$ assuming that $1/\ln{n}$ is the asymptotic density of the set of all prime numbers $\mathbb{P{}}$ as $n \rightarrow \infty$. We approach this task by alternately adding and subtracting the prime number densities (or ``probabilities") of the prime number subsequences of increasing order, also known as "superprimes" \cite{BroughanBarnett2009}, to arrive at values for the asymptotic densities for $\mathbb{P{'}}$ and $\mathbb{P{''}}$.  We begin by recalling \cite{May2020} that the prime number subsequence $\mathbb{P{'}}$ is formed by the alternating series

\begin{equation*}
{\mathbb{P{'}}} = {\left\lbrace{(-1)^{n-1}}\left\lbrace{p^{(n)}}\right\rbrace\right\rbrace}_{n=1}^\infty = \left\lbrace{p^{(1)}}\right\rbrace - \left\lbrace{p^{(2)}}\right\rbrace + \left\lbrace{p^{(3)}}\right\rbrace -  ... 
\end{equation*}

\noindent where

\begin{equation*}
\left\lbrace{p^{(k)}}\right\rbrace = \left\lbrace{{p_{p_{._{._{._{p_n}}}}}}}\right\rbrace \; \text{($p$ ``$k$" times)}.
\end{equation*}

\

\noindent Broughan and Barnett have shown \cite{BroughanBarnett2009} that for the general case of higher-order ``superprimes" ${p_{p_{._{._{._{p_{k}}}}}}}$, the asymptotic density is approximately

\begin{equation*}
{\dfrac{n}{p_{p_{._{._{._{p_n}}}}}} \sim \dfrac{n}{n\,(\ln{n})^k} \sim \dfrac{1}{(\ln{n})^k} } \;\;\; 
\end{equation*}

\noindent for large $n \in \mathbb{N{}}$. Now, assuming that $1/\ln{n}$ is the asymptotic density for the set of all prime numbers $\mathbb{P{}}$, we derive an expression for the density $d'$ for the prime number subsequence $\mathbb{P{'}}$ at $\infty$. We begin with the geometric series

\begin{equation*}
{S=1-x+x^2-x^3+x^4-x^5+... = \dfrac{1}{1+x} \;\;\; (\lvert x \rvert < 1)}.
\end{equation*}

\noindent Then let

\begin{align*}
{T'} \;\; & {= 1 - S}
\end{align*}

\noindent so that

\begin{align*}
{T'} \;\; & {= x - x^2 + x^3 - x^4 + x^5+...} \\
& {= \dfrac{-1}{1+x}+1} \\
& {= \dfrac{x}{1+x}}.
\end{align*}

\noindent Now substitute $\dfrac{1}{\ln{n}}$ for $x$ to get

\begin{equation*}
{\dfrac {\dfrac{1}{\ln{n}}}{1+\dfrac{1}{\ln{n}}} = \dfrac{1}{\ln{n}+1}}
\end{equation*}

\noindent so that we have

\begin{align}
{d'}  & { \approx \dfrac{1}{\ln{n}}-\dfrac{1}{(\ln{n})^2}+\dfrac{1}{(\ln{n})^3}-\dfrac{1}{(\ln{n})^4}+...} \\
& {= \dfrac{1}{\ln{n}+1}}.  \label{eq:6}
\end{align}

\noindent Similarly, we derive the asymptotic density for the prime number subsequence $\mathbb{P{''}}$.  When we set

\begin{align*}
{T''} \;\; & {= S - (1 - x)}
\end{align*}

\noindent we have

\begin{align*}
{T''} \;\; & {= S - (1 - x)} \\
& {= {x^2 - x^3 + x^4 - x^5 +...}} \\
& {= \dfrac{1}{1+x}-1+x} \\
& {= \dfrac{x^2}{1+x}}.
\end{align*}

\noindent Now substitute $\dfrac{1}{\ln{n}}$ for $x$ to get

\begin{equation*}
{\dfrac {\left(\dfrac{1}{\ln{n}}\right)^2}{1+\dfrac{1}{\ln{n}}} = \dfrac{1}{\ln{n}(\ln{n}+1)}}
\end{equation*}

\noindent so that

\begin{align}
{d''}  & { \approx \dfrac{1}{(\ln{n})^2}-\dfrac{1}{(\ln{n})^3}+\dfrac{1}{(\ln{n})^4}-\dfrac{1}{(\ln{n})^5}+...} \\
& {= \dfrac{1}{\ln{n}(\ln{n}+1)}}.  \label{eq:7}
\end{align}

\noindent Based on our assumption that $1/\ln{n}$ is the asymptotic density of the set of all prime numbers $\mathbb{P{}}$ as $n \rightarrow \infty$, Eqs. \ref{eq:6} and \ref{eq:7} provide us with the densities (or probabilities of occurrence) of the primes in the complementary sets $\mathbb{P{'}}$ and $\mathbb{P{''}}$, respectively, as $n$ approaches $\infty$.  Thus, the average gap size $g'$ between prime numbers in the subsequence $\mathbb{P{'}}$ on the natural number line as $n \rightarrow \infty$ is the inverse of the density $d'$ of $\mathbb{P{'}}$ such that

\begin{align*}
{{g'}} = {\frac{1}{d'}} & \approx {\dfrac{1}{\dfrac{1}{\ln{n}}-\dfrac{1}{(\ln{n})^2}+\dfrac{1}{(\ln{n})^3}-\dfrac{1}{(\ln{n})^4}+...}} \\
& = {\ln{n}+1}.
\end{align*}

\noindent Similarly, the average gap size $g''$ between prime numbers in the subsequence $\mathbb{P{''}}$ on the natural number line as $n \rightarrow \infty$ is the inverse of the density $d''$ of $\mathbb{P{''}}$ such that

\begin{align*}
{{g''}} = {\frac{1}{d''}} & \approx {\dfrac{1}{\dfrac{1}{(\ln{n})^2}-\dfrac{1}{(\ln{n})^3}+\dfrac{1}{(\ln{n})^4}-\dfrac{1}{(\ln{n})^5}+...}} \\
& = {\ln{n}(\ln{n}+1)}.
\end{align*}

\noindent Since it has been shown via the sieving operation \cite{May2020} that the prime number subsequence $\mathbb{P{'}}$ has fewer primes than the set of all prime numbers $\mathbb{P{}}$, it intuitively follows that the average gap size for $\mathbb{P{'}}$ will always be larger than the gap size for $\mathbb{P{}}$ and that the larger gap size for $\mathbb{P{'}}$ results from omitting the count of the prime numbers $\mathbb{P{''}}$ on $\mathbb{N}$.

\section{${\pi{'(x)}}$ and ${\pi{''(x)}}$}

\noindent We have shown that when we remove the prime number subsequence $\mathbb{P{''}}$ from the set of all prime numbers $\mathbb{P{}}$, we create the prime number subsequence $\mathbb{P{'}}$ \cite{May2020}.  Thus, we define the prime number count for the sequences $\mathbb{P{'}}$ and $\mathbb{P{''}}$ up to $x$ as

\begin{center}
${\pi{'(x)}} = \vert \mathbb{P'}(x) \vert$
\end{center}

\noindent and

\begin{center}
${\pi{''(x)}} = \vert \mathbb{P''}(x) \vert$
\end{center}

\noindent where $\vert \mathbb{P'}(x) \vert$ and $\vert \mathbb{P''}(x) \vert$ represent the cardinality of the prime number subsequences $\mathbb{P{'{}}}$ and $\mathbb{P{''{}}}$ up to $x$.  However, since neither ${\pi{'(x)}}$ nor ${\pi{''(x)}}$ have been shown up to this point to be calculable without manually counting each term up to $x$, we will begin by generating an estimate of the count ${\pi{(x)}}$ of set of all primes $\mathbb{P{}}$ up to $x$ via the Inclusion-Exclusion Principle and then perform an operation on that result to reduce the count of all primes down to ${\pi{'(x)}}$ and ${\pi{''(x)}}$.

\newpage

\section{${\pi{(x)}}$ via the Inclusion-Exclusion Principle}

\noindent To calculate $\pi{(x)}$, we invoke the Inclusion-Exclusion Principle \cite{McKernan2017} \cite{Leveque1977}. Let $r$ represent the number of primes less than $\sqrt{x}$. Then let $P = \lbrace n \in N \vert 1 < n \leq x \rbrace$ such that $n$ is not a multiple of $p_1,p_2,...,p_r$. If $A(x,r)$ represents the cardinality of $P$, then it follows that the number of primes $\leq x$ is

\begin{center}
\noindent ${\pi{(x)}} \leq r + A(x,r)$.
\end{center}

\noindent Now, let $M_i$ be the set of integers from $1$ to $n$ which are multiples of $p_i$, and let $M_{ij}$ be the set of integers from $1$ to $n$ that are multiples of both $p_i$ and $p_j$. Then,

\begin{center}
\noindent ${M_{ij} = M_i \cap M_j}$
\end{center}

\noindent so that

\begin{center}
\noindent ${\vert M_{i} \vert = \lfloor {\dfrac {x}{p_i}} \rfloor}$ \; and \; ${\vert M_{ij} \vert = \lfloor {\dfrac {x}{p_i p_j}} \rfloor}$.
\end{center}

\noindent Then it follows by the Inclusion-Exclusion Principle that

\begin{equation}
{A(x,r) = \lfloor x \rfloor - {\sum_{i=1}^{r}} \lfloor {\dfrac{x}{p_i}} \rfloor + {\sum_{i<j\leq r}^{r}} \lfloor {\dfrac{x}{p_i p_j}} \rfloor - \ldots +(-1)^r \lfloor {\dfrac{x}{p_1 p_2 \cdots p_r}} \rfloor}. \label{eq:8}
\end{equation}

\noindent If we approximate the RHS of Eq. \ref{eq:8} by ignoring the round-downs, then we have

\begin{equation*}
{ x  - {\sum_{i=1}^{r}}  {\dfrac{x}{p_i}}  + {\sum_{i<j\leq r}^{r}}  {\dfrac{x}{p_i p_j}} - \ldots +(-1)^r  {\dfrac{x}{p_1 p_2 \cdots p_r}} }
\end{equation*}

\noindent with an error of at most

\begin{equation*}
{1 + \binom{r}{1} + \binom{r}{2} + \ldots + \binom{r}{r} = 2^r}.
\end{equation*}

\noindent Thus, we now have for our estimate of the number of primes less than or equal to $x$ as

\begin{equation}
{\pi{(x)}} \leq r + {x \cdot \prod_{i=1}^{r} \left(1 - \dfrac{1}{p_i}\right) + 2^r}. \label{eq:9}
\end{equation}

\noindent We now want to choose $r$ relatively small compared to $x$. In order to do so, we need a good estimate of the coefficient of $x$ in the middle term on the RHS of \ref{eq:9} in terms of $r$.

\begin{theorem} \label{1}

If $x \geq 2$, then ${\prod_{p\leq x} \left( 1 - \dfrac{1}{p} \right)} < \dfrac{1}{\ln{x}}$.

\end{theorem}

\begin{proof}

\begin{equation*}
{\prod_{p \leq x} \dfrac{1}{1 - \dfrac{1}{p}}} = {\prod_{p \leq x} \left(1 + \dfrac{1}{p} + \dfrac{1}{p^2} +\ldots \right)}.
\end{equation*}

\noindent Now,

\begin{equation*}
{\prod_{p \leq x} \dfrac{1}{1 - \dfrac{1}{p}}} > {\sum_{k=1}^n\dfrac{1}{k}} > {\int_{1}^{\lceil x \rceil} \dfrac{du}{u}} > \ln{x}.
\end{equation*}

\begin{equation*}
{{\therefore}} \;\; {\prod_{p \leq x} \left(1 - \dfrac{1}{p}\right)} < \dfrac{1}{\ln{x}} \Rightarrow {\prod_{i=1}^{r} \left(1 - \dfrac{1}{p_{i}}\right)} < \dfrac{1}{\ln{p_{r}}}.
\end{equation*}

\end{proof}

\noindent We now have as our estimate of $\pi{(x)}$,

\begin{equation}
{{\pi{(x)}} \leq r + \dfrac{x}{\ln{p_r}}  + 2^r}. \label{eq:10}
\end{equation}

\section{Estimating ${\pi{''(x)}}$}

\noindent In order to calculate an estimate of $\pi{''(x)}$, we begin by taking a look at the coefficient of $x$ in the middle term on the RHS of \ref{eq:9}. We can write that coefficient as

\begin{equation}
{\prod_{i=1}^{r} \left(1 - \dfrac{1}{p_i}\right) = \prod_{i=1}^{r'} \left(1 - \dfrac{1}{p'_i}\right) \cdot \prod_{i=1}^{r''} \left(1 - \dfrac{1}{p''_i}\right) } \label{eq:11}
\end{equation}

\noindent where $r'$ is the number of $p' <$ the number of the first $r$ primes $\leq \sqrt{x}$, and $r''$ is the number of $p'' <$ the number of the first $r$ primes $\leq \sqrt{x}$ such that $r'+r''=r$.  We found that one cannot simply divide the product on the LHS of Eq. \ref{eq:11} by either product on the RHS of Eq. \ref{eq:11} and expect the quotient to represent a pure count of $p'$ or $p'' \leq x$.  Therefore, we must approach the problem from a different direction; i.e., we must find another way to reduce the coefficient of $x$ on the RHS of \ref{eq:9} such that the estimate will leave the count of $p''$ \underline{only} with no $p'$ and no composites remaining when the coefficient is multiplied by $x$. Hence, we model the inequality in Theorem \ref{1} as

\begin{equation}
{\prod_{i=1}^{r} \left(1 - \dfrac{1}{p_i}\right) < \dfrac{1}{\ln{p_r}} = \dfrac{1}{{\ln^j{p_r}}\cdot {\ln^k{p_r}}}                    }. \label{eq:12}
\end{equation}

\noindent It was found that if the last term on the RHS of Eq. \ref{eq:12} is multiplied by either $\ln^j{p_r}$ or $\ln^k{p_r}$, then the resultant value is greater than

\begin{equation*}
{ \dfrac{1}{\ln{p_r}}  }
\end{equation*}

\noindent which is counterintuitive to our proof that the complementary prime number subsequences $\mathbb{P'}$ and $\mathbb{P''}$ add to form the complete set of prime numbers $\mathbb{P}$. Multiplying the last term of Eq. \ref{eq:12} by either $\ln^j{p_r}$ or $\ln^k{p_r}$ actually increases the cardinality of $\mathbb{P'}(x)$ or $\mathbb{P''}(x)$ to be greater than the cardinality of the entire set of prime numbers $\mathbb{P}(x)$ when that coefficient is multiplied by $x$ on the RHS of \ref{eq:9}. Hence our motivation to approach the solution from a different direction. In that light, it was found that if we let 

\begin{equation}
{ \dfrac{1}{{\ln^j{p_r}}\cdot {\ln^k{p_r}}} =  \dfrac{1}{\ln{p_r}} \cdot \biggl[ {j+k} \biggr]} \label{eq:13}
\end{equation}

\noindent we can then subtract 

\begin{equation*}
{ \dfrac{j}{\ln{p_r}}} \;\;\; \text{or} \;\;\; { \dfrac{k}{\ln{p_r}}}
\end{equation*}

\noindent from the LHS of Eq. \ref{eq:13} (or from the RHS of Eq. \ref{eq:12}) and obtain the proper coefficient to multiply times $x$ on the RHS of \ref{eq:10} to obtain the correct estimate of the quantity of $\mathbb{P'}(x)$ or $\mathbb{P''}(x)$ depending upon which of these quantities is subtracted from the RHS of Eq. \ref{eq:13}. So the task at hand is to find a $j$ and a $k$ that will satisfy both sides of Eq. \ref{eq:13}. To that end, it is seen in Eq. \ref{eq:13} that $j$ and $k$ must sum to unity on both sides of the equation to make this approach work. In order to do so, we recall the asymptotic densities that we derived earlier for ${\pi{'(x)}}$ and ${\pi{''(x)}}$ as

\begin{align*}
{\pi'{(x)} \sim \dfrac{1}{\ln{n}+1}} \;\; \text{and} \;\; {\pi''{(x)} \sim \dfrac{1}{\ln{n}(\ln{n}+1)}}.
\end{align*}

\noindent Now, since $j+k=1$ must hold true to satisfy Eq. \ref{eq:13}, we let

\begin{equation*}
{j = \dfrac{{\dfrac{1}{\ln{n}(\ln{n}+1)}}}{\dfrac{1}{\ln{n}}}} = \dfrac{\ln{n}}{\ln{n}(\ln{n}+1)}
\end{equation*}

\noindent and

\begin{equation*}
{k = \dfrac{{\dfrac{1}{\ln{n}+1}}}{\dfrac{1}{\ln{n}}}} = \dfrac{\ln{n}}{\ln{n}+1}
\end{equation*}

\noindent such that $j=$ the ratio of the asymptotic density of the prime subsequence $\mathbb{P''}$ divided by the asymptotic density of the set of all primes $\mathbb{P}$; and $k=$ the ratio of the asymptotic density of the prime subsequence $\mathbb{P'}$ to the asymptotic density of all primes $\mathbb{P}$. We now introduce a lemma:

\begin{lemma}
$ \text{For} \; x>1, j + k = 1.$
\begin{proof}
\begin{align*}
{j+k} \;\; & {= \dfrac{\ln{x}}{\ln{x}+1} + \dfrac{\ln{x}}{\ln{x}(\ln{x}+1)}} \\
& {= \dfrac{\ln{x}(\ln{x}+1)+\ln{x}\ln{x}(\ln{x}+1)}{\ln{x}(\ln{x}+1)^2}} \\
& {= \dfrac{\ln{x}(\ln{x}+1)(\ln{x}+1)}{\ln{x}(\ln{x}+1)^2}} \\
& {= 1}.
\end{align*}
\end{proof}
\end{lemma}

\noindent Since $j + k = 1$ is valid in the lemma for all $x > 1$, we have established that

\begin{align*}
{\dfrac{1}{{\ln^j{p_r}}\cdot {\ln^k{p_r}}}} \;\; & {=\dfrac{1}{\ln{p_r}} \cdot \biggl[ {j+k} \biggr]} \\
& {= \dfrac{1}{\ln{p_r}}\cdot \biggl[ \dfrac{\ln{p_r}}{\ln{p_r}(\ln{p_r}+1)} + \dfrac{\ln{p_r}}{\ln{p_r}+1}\biggr]}. \\
\end{align*}

\noindent Thus, in harmony with the lemma, we have

\begin{equation}
{\pi{''(x)} \leq r'' + \dfrac{x}{\ln{p_r}(\ln{p_r}+1)} + 2^{r''} } \label{eq:14}
\end{equation}

\noindent where $r'' =$ the number of $p'' \leq p_r$ and $2^{r''}=$ the maximum error resulting from the main term. We now proceed with our estimate of \ref{eq:14}. We know that

\begin{equation*}
{r'' \leq {\lfloor{\dfrac{r}{2}}\rfloor} \;\;\; \text{for} \;\;\; p'' > 2}
\end{equation*}

\noindent because it was proven that there are fewer primes in the subsequence $\mathbb{P''}$ than in the complementary prime subsequence $\mathbb{P'}$ (recall Eq. \ref{eq:5}).  Thus,

\begin{align*}
{\pi{''(x)}} & {\leq r'' + \dfrac{x}{\ln{p_r}(\ln{p_r}+1)} + 2^{r''}} \;\;\;\;\;\;\;\;\;\;\;\;\;\;\;\;\;\;\; (\ref{eq:14}) \\ 
& {< r'' + \dfrac{x}{\ln{r}(\ln{r}+1)} + 2^{r''}} \;\;\;\;\;\;\;\;\;\;\;\;\;\;\;\;\;\;\; \left(r < p^r\right) \\
& {\leq\lfloor{\dfrac{r}{2}}\rfloor} + {\dfrac{x}{\ln{r}(\ln{r}+1)} + 2^{\lfloor{\frac{r}{2}}\rfloor}} \;\;\;\;\;\;\;\;\; \left({\lfloor{\dfrac{r}{2}}\rfloor} \geq r''\right) \\
&  {< \dfrac{x}{\ln{r}(\ln{r}+1)} + 2^{\lfloor{\frac{r}{2}}\rfloor+1}} \;\;\;\;\;\;\;\;\;\;\;\;\; \left(2^{\lfloor{\frac{r}{2}}\rfloor} >  {\lfloor{\dfrac{r}{2}}\rfloor}\right) \\
&  {< \dfrac{x}{\ln{r}(\ln{r}+1)} + 2^{{\frac{r}{2}}+1}} \;\;\;\;\;\;\;\;\;\;\;\;\;\;\;\;\;\;\; \left({\dfrac{r}{2}} >  {\lfloor{\dfrac{r}{2}}\rfloor}\right). \\
\end{align*}

\noindent Now, let

\begin{equation*}
{r = x^m \;\;\; \text{such that} \;\;\; m = {\frac{1}{c \cdot \ln{\ln{x}}}}} \;\;\; \text{for some positive constant} \;\; c.
\end{equation*}

\noindent We now have

\begin{align*}
{\pi{''(x)}} & {< \dfrac{x}{\ln{x^m} ({\ln{x^m}+1})} + 2^{\frac{1}{2}(x^m + 2)}} \;\;\;\;\;\;\;\;\;\;\;\;\;\;\;\;\;\;\; \left(  m = {\frac{1}{c \cdot \ln{\ln{x}}}} \right) \\ 
& {= \dfrac{x}{   \left( \dfrac{\ln{x}}{c \cdot \ln{\ln{x}}} \right)^2 +  \dfrac{\ln{x}}{c \cdot \ln{\ln{x}}}   } + 2^{\frac{1}{2}(x^m + 2)}} \\
& {= \dfrac{x}{   \dfrac{c\cdot\ln{\ln{x}} \cdot \ln^2{x} + (c\cdot\ln{\ln{x}})^2 \cdot \ln{x}}{(c\cdot\ln{\ln{x}})^3}}       + 2^{\frac{1}{2}(x^m + 2)}} \\
&  {= x \cdot \dfrac{(c \cdot \ln{\ln{x}})^2}{ \ln^2{x} +  c\cdot\ln{\ln{x}} \cdot \ln{x}}       + 2^{\frac{1}{2}(x^m + 2)}} \\
&  {< C \cdot x \cdot \dfrac{( \ln{\ln{x}})^2}{ \ln^2{x} +  c\cdot\ln{\ln{x}} \cdot \ln{x}}   + 2^{\frac{1}{2}(x^m + 2)}    }. \\
\end{align*}

\noindent Since $\ln{\ln{x}} \cdot \ln{x} <  c \cdot \ln{\ln{x}} \cdot \ln{x}$ for $c\geq1$, and since $2^{\frac{1}{2}(x^m + 2)} \ll$ than the main term when $c \geq 5$, we finally arrive at

\begin{equation}
{\pi''{(x)} < C \cdot x \cdot \dfrac{( \ln{\ln{x}})^2}{ (\ln{x})^2 +  \ln{\ln{x}} \cdot \ln{x}}}. \label{eq:15}
\end{equation}

\noindent Thus, we see that for some positive constant $C < +\infty$, the sum of the reciprocals of the infinite subsequence of prime numbers $\mathbb{P''}$ converges, and this is confirmed when we compare \ref{eq:15} to the count of $p_2 \leq x$ (see \ref{eq:20}). 

\section{Estimating ${\pi{'(x)}}$}

\noindent In order to calculate an estimate of $\pi{'(x)}$, we invoke the lemma and begin with 

\begin{equation}
{\pi{'(x)} \leq r' + \dfrac{x}{\ln{p_r}+1} + 2^{r'} }. \label{eq:16}
\end{equation}

\noindent Proceeding as before, we know that

\begin{equation*}
{r' \leq {\lfloor {r} \rfloor } \;\;\; \text{for} \;\;\; p' \geq 2}
\end{equation*}

\noindent since there are fewer primes in the subsequence $\mathbb{P'}$ than in the set of all prime numbers $\mathbb{P}$.  Thus,

\begin{align*}
{\pi{'(x)}} & {\leq r' + \dfrac{x}{\ln{p_r}+1} + 2^{r'}} \;\;\;\;\;\;\;\;\;\;\;\;\;\;\;\;\;\;\;\;\;\;\; (\ref{eq:16}) \\ 
& {< r' + \dfrac{x}{\ln{r}+1} + 2^{r'}} \;\;\;\;\;\;\;\;\;\;\;\;\;\;\;\;\;\;\; \left( {r < p^r} \right) \\
& {\leq\lfloor{{r}}\rfloor} + {\dfrac{x}{\ln{r}+1} + 2^{\lfloor{{r}}\rfloor}} \;\;\;\;\;\;\;\;\;\;\;\; \left({\lfloor{{r}}\rfloor} \geq r'\right) \\
&  {< \dfrac{x}{\ln{r}+1} + 2^{\lfloor{{r}}\rfloor+1}} \;\;\;\;\;\;\;\;\;\;\;\;\;\; \left(2^{\lfloor{{r}}\rfloor} >  {\lfloor{{r}}\rfloor}\right) \\
&  {< \dfrac{x}{\ln{r}+1} + 2^{{{r}}+1}} \;\;\;\;\;\;\;\;\;\;\;\;\;\;\;\;\;\;\;\;\;\; \left({r} >  {\lfloor{{r}}\rfloor}\right). \\
\end{align*}

\noindent Now, let

\begin{equation*}
{r = x^m \;\;\; \text{such that} \;\;\; m = {\frac{1}{c \cdot \ln{\ln{x}}}}} \;\;\; \text{for some positive constant} \;\; c.
\end{equation*}

\noindent We now have

\begin{align*}
{\pi{'(x)}} & {< \dfrac{x}{{\ln{x^m}+1}} + 2^{x^m + 1}} \;\;\;\;\;\;\;\;\;\;\;\;\;\;\;\;\;\;\; \left( m = {\frac{1}{c \cdot \ln{\ln{x}}}}\right) \\ 
& {= \dfrac{x}{    \dfrac{\ln{x}}{c \cdot \ln{\ln{x}}}  +  1    } + 2^{x^m + 1}} \\
& {= \dfrac{x}{   \dfrac{ \ln{x} + c\cdot\ln{\ln{x}}}{c\cdot\ln{\ln{x}}}  }       + 2^{x^m + 1}} \\
&  {= x \cdot \dfrac{c \cdot \ln{\ln{x}}}{ \ln{x} + c\cdot\ln{\ln{x}}}       +  2^{x^m + 1}} \\
&  {< C \cdot x \cdot \dfrac{ \ln{\ln{x}}}{ \ln{x} +  c\cdot\ln{\ln{x}}}   +  2^{x^m + 1}    }. \\
\end{align*}

\noindent Since $\ln{\ln{x}}  <  c \cdot \ln{\ln{x}}$ for $c\geq1$, and since $2^{x^m + 1} \ll$ than the main term when $c \geq 5$, we finally arrive at

\begin{equation*}
{\pi'{(x)} < C \cdot x \cdot \dfrac{ \ln{\ln{x}}}{ \ln{x} +\ln{\ln{x}} }}. \label{eq:17}
\end{equation*}

\noindent Since we've shown that $\mathbb{P''}$ is a small set in that the infinite sum of its reciprocals converges, and since it is known that the sum of the reciprocals of the set of all prime numbers $\mathbb{P}$ diverges, we can deduce from the relation 

\begin{equation*}
{\mathbb{P{}}}={\mathbb{P{'}}+\mathbb{P{''}}}
\end{equation*}

\noindent that the prime number subsequence $\mathbb{P'}$ is a large set and that the infinite sum of its reciprocals diverges.

\section{Estimating ${\pi{_2(x)}}$}

\noindent We now take a look at how the twin prime count ${\pi{_2(x)}}$ can be estimated using the technique heretofore disclosed. If we assume that

\begin{equation*}
{{\pi{_2(x)}} \sim \dfrac{C}{\ln^2{x}}} \label{eq:18}
\end{equation*}

\noindent for some positive constant $C$ \cite{HardyLittlewood1923}, we can then model $j$ found on the RHS of Eq. \ref{eq:13} as

\begin{equation*}
{j = \dfrac{{\dfrac{C}{\ln^2{p_r}}}}{\dfrac{1}{\ln{p_r}}} = C \cdot \dfrac{1}{\ln{p_r}}}
\end{equation*}

\noindent and we can model $k$ found on the RHS of Eq. \ref{eq:13} as 

\begin{equation*}
{k = 1 - C \cdot \dfrac{1}{\ln{p_r}}}
\end{equation*}

\noindent such that $j=$ the ratio of the asymptotic density of the twin prime subsequence $\mathbb{P}_2$ divided by the asymptotic density of the set of all primes $\mathbb{P}$; and $k=$ the ratio of the asymptotic density of the set of remaining prime numbers [$ \mathbb{P} - \mathbb{P}_2$] to the asymptotic density of all primes $\mathbb{P}$ so that $j+k=1$. We can then model $\pi{_2(x)}$ as

\begin{align}
{\pi{_2(x)}} & {\leq r_2 + x \cdot \dfrac{1}{\ln{p_r}} \cdot \left[ C \cdot \dfrac{1}{\ln{p_r}} \right] + 2^{r_2}}  \\ \label{eq:19}
& {= r_2 + x \cdot   C \cdot \dfrac{1}{\ln^2{p_r}}  + 2^{r_2}} 
\end{align}

\noindent where $r_2=$ the number of $p_2 \leq p_r$ and $2^{r_2}=$ the maximum error resulting from the main term. Similar to $r''$, we know that

\begin{equation*}
{r_2 \leq {\lfloor{\dfrac{r}{2}}\rfloor} \;\;\; \text{for} \;\;\; p_2 > 2}
\end{equation*}

\noindent because there are fewer twin primes $\mathbb{P}_2$ than half the count of all prime numbers $\mathbb{P}$.  Thus,

\begin{align*}
{\pi{_2(x)}} & {\leq r_2 + x \cdot \dfrac{C}{\ln^2{p_r}} + 2^{r_2}} \;\;\;\;\;\;\;\;\;\;\;\;\;\;\;\;\;\;\;\;\;\;\;\; \left( \ref{eq:19} \right) \\\ 
& {< r_2 + x \cdot \dfrac{C}{\ln^2{r}} + 2^{r_2}} \;\;\;\;\;\;\;\;\;\;\;\;\;\;\;\;\;\;\;\;\; \left(r < p^r\right) \\
& {\leq\lfloor{\dfrac{r}{2}}\rfloor} + {x \cdot \dfrac{C}{\ln^2{r}} + 2^{\lfloor{\frac{r}{2}}\rfloor}} \;\;\;\;\;\;\;\;\;\;\; \left( {\lfloor{\dfrac{r}{2}}\rfloor} \geq r_2\right) \\
&  {< x \cdot \dfrac{C}{\ln^2{r}} + 2^{\lfloor{\frac{r}{2}}\rfloor+1}} \;\;\;\;\;\;\;\;\;\;\;\;\;\;\; \left(2^{\lfloor{\frac{r}{2}}\rfloor}  >  {\lfloor{\dfrac{r}{2}}\rfloor}\right) \\
&  {< x \cdot \dfrac{C}{\ln^2{r}} + 2^{{\frac{r}{2}}+1}} \;\;\;\;\;\;\;\;\;\;\;\;\;\;\;\;\;\;\;\;\;\; \left({\dfrac{r}{2}} >  {\lfloor{\dfrac{r}{2}}\rfloor} \right). \\
\end{align*}

\noindent Now, let

\begin{equation*}
{r = x^m \;\;\; \text{such that} \;\;\; m = {\frac{1}{c \cdot \ln{\ln{x}}}}} \;\;\; \text{for some positive constant} \;\; c. \;\; \text{We now have}
\end{equation*}

\begin{align*}
{\pi{_2(x)}} & {< x \cdot \dfrac{C}{\ln^2{x^m}} + 2^{\frac{1}{2}(x^m + 2)}}  \\ 
& {= x \cdot \dfrac{C}{   \left( \dfrac{\ln{x}}{c \cdot \ln{\ln{x}}} \right)^2    } + 2^{\frac{1}{2}(x^m + 2)}} \\
&  {= x \cdot \dfrac{C \cdot (c \cdot \ln{\ln{x}})^2}{ (\ln{x})^2}       + 2^{\frac{1}{2}(x^m + 2)}} \\
&  {= C \cdot x \cdot \dfrac{( \ln{\ln{x}})^2}{ (\ln{x})^2 }} + 2^{\frac{1}{2}(x^m + 2)}. \\
\end{align*}

\noindent And since  $2^{\frac{1}{2}(x^m + 2)} \ll$ than the main term for $c \geq 5$, we arrive at

\begin{equation}
{\pi_2{(x)} < C \cdot x \cdot \dfrac{( \ln{\ln{x}})^2}{ (\ln{x})^2 }}. \label{eq:20}
\end{equation}

\noindent Thus, it is confirmed via this approach that for some positive constant $C < +\infty$, the sum of the reciprocals of the twin primes $\mathbb{P}_2$ converges. Further, when we compare the inequality \ref{eq:20} with the inequality for $\mathbb{P}''$ in \ref{eq:15}, we see that the count of $p'' \leq x$, or $\pi''{(x)}$, is less than the count of twin primes $p_2 \leq x$, or $\pi_2{(x)}$.

\section{Mathematica calculations}

\noindent Mathematica \cite{Mathematica} was programmed to calculate the sum of the reciprocals of $\mathbb{P''}$ and $\mathbb{P}_2$ for various ranges of $x$ up to $10E6$, and a table of the computations appears below. Table \ref{table:reciprocal} reveals that through the ranges of $x$ calculated, the sum of the reciprocals of $p''$ is smaller than the sum of the reciprocals for the twin primes $p_2$, both of which converge at $\infty$.

\begin{table}[ht]
\centering
\caption{$p''(x)$ and $p_2(x)$ reciprocal sums}
\label{table:reciprocal}
\begin{tabular}{|c|c|c|c|c|c|c|c|}
\hline 
$x$ & $\sum{\frac{1}{p''(x)}}$ & $\sum{\frac{1}{p_2(x)}}$  \\
\hline 
1E02 & 0.534430 & 1.28989    \\
\hline 
1E03 & 0.606479 & 1.40995    \\
\hline 
1E04 & 0.644283 & 1.47370    \\
\hline 
1E05 & 0.668046 & 1.51443    \\
\hline 
1E06 & 0.683968 & 1.54268    \\
\hline 
2E06 & 0.687789 & 1.54950    \\
\hline 
3E06 & 0.689858 & 1.55321    \\
\hline 
4E06 & 0.691258 & 1.55573    \\
\hline 
5E06 & 0.692310 & 1.55763    \\
\hline 
6E06 & 0.693139 & 1.55915    \\
\hline 
7E06 & 0.693834 & 1.56040    \\
\hline 
8E06 & 0.694421 & 1.56148    \\
\hline 
9E06 & 0.694932 & 1.56240    \\
\hline 
10E6 & 0.695379 & 1.56322    \\
\hline 
\end{tabular}
\end{table}

\newpage

\section{Conclusion}

\noindent In this paper, we applied the Inclusion-Exclusion Principle to the complementary prime number subsequences $\mathbb{P}'$ and $\mathbb{P}''$ to derive the respective prime counting functions $\pi'{(x)}$ and $\pi''{(x)}$ to determine whether these subsequences form a small set or a large set and thus whether the infinite sum of the inverse of their terms converges or diverges. In this study, we concluded that the sum of the reciprocals of the prime number subsequence $\mathbb{P}'$ diverges, similar to that for the set of all prime numbers $\mathbb{P}$, while the sum of the reciprocals of the prime number subsequences $\mathbb{P}''$ converges, similar to that for the set of all twin primes $\mathbb{P}_2$.

\end{document}